\documentclass[twoside,12pt,reqno]{amsart}
\usepackage{hyperref}
\hypersetup{colorlinks,linkcolor=myurlcolor,citecolor=myurlcolor,urlcolor=myurlcolor}
\usepackage{graphics,epstopdf,graphicx,amsthm,amsmath,amssymb,mathptmx,braket,colortbl,color,bm,framed,mathrsfs}
\usepackage[T1]{fontenc}
\usepackage[up]{subfigure}
\usepackage{tikz}
\usepackage{float}
\usepackage{amsfonts}
\definecolor{myurlcolor}{rgb}{0,0,0.9}

\newcommand{\be}{\begin{equation}}
\newcommand{\ee}{\end{equation}}
\newcommand{\beq}{\begin{eqnarray}}
\newcommand{\eeq}{\end{eqnarray}}
\newcommand{\beqs}{\begin{eqnarray*}}
\newcommand{\eeqs}{\end{eqnarray*}}

\newcommand{\inner}[2]{\langle #1 , #2\rangle}

\renewcommand\leq{\leqslant}
\renewcommand\geq{\geqslant}

\newcommand{\norm}[1]{\left\lVert #1 \right\rVert}

\newcommand{\real}{\mathbb{R}}

\theoremstyle{plain}
\newtheorem{thm}{Theorem}
\newtheorem{lem}[thm]{Lemma}
\newtheorem{prop}[thm]{Proposition}
\newtheorem{cor}[thm]{Corollary}
\newtheorem{Def}[thm]{Definition}

\title{Duality of Graph Invariants}
\author{Kaifeng Bu$^{+*}$}
\email{kfbu@fas.harvard.edu}
\author{Weichen Gu$^{\times *}$}
\email{guweichen14@mails.ucas.ac.cn}
\author{{{Arthur Jaffe}}$^{*}$}
\email{arthur\_jaffe@harvard.edu}
\address[*]{Harvard University, Cambridge, MA 02138, USA}
\address[+]{School of Mathematical Science , Zhejiang University,  Hangzhou  310027, China}
\address[$\times$]{Academy of Mathematics and System Science, CAS, Beijing 100080, China}

\begin{document}

\begin{abstract}
We  study a new set of duality relations between weighted, combinatoric invariants of a graph  $G$.
%: weighted independence numbers, weighted Lov\'{a}sz numbers, and weighted fractional packing numbers.   
The dualities arise from a non-linear transform $\mathfrak{B}$, acting on the weight function $p$. We define $\mathfrak{B}$ on a  space of real-valued functions $\mathcal{O}$ and investigate  its properties. 
We show that three invariants (weighted independence number, weighted Lov\'{a}sz number, and weighted fractional packing number) are fixed points of  $\mathfrak{B}^{2}$, but the weighted Shannon capacity is not. We interpret these invariants in the study of quantum non-locality.
\end{abstract}

\maketitle
\setcounter{tocdepth}{1}
\tableofcontents

\section{Introduction}
We  study a new set of duality relations between weighted, combinatoric invariants of a graph  $G$, with a fixed set of vertices, but with a variable weight function $p$.  These  dualities arise from a non-linear transform $\mathfrak{B}$, acting on the weight function $p$. 
This is a different duality from the one introduced in~\cite{Acin2017}.
We define $\mathfrak{B}$ on a  space of real-valued functions $\mathcal{O}$ and investigate  its properties. 

In particular,  we give duality relations for the weighted independence number, the weighted Lov\'{a}sz number, and the weighted fractional packing number. We  state these dualities and  prove them  in \S\ref{sect:DualityProof}.  In  \S\ref{sect:Interpretation} we discuss the relations of our results to the study of relative entropy, entanglement, and quantum coherence.  In \S\ref{sec:math_prop} we elaborate the definition  and give detailed properties of the transformation~$\mathfrak{B}$.

\subsection{Background}
Graph invariants have been widely used in communication theory. For example, Shannon defined zero-error channel capacity
by the independence number of graphs, which is used to  quantify the maximal rate of information sent by the channel without error \cite{Shannon56}.
Moreover, the fractional packing number was also proposed to quantify the zero-error channel capacity with feedback\cite{Shannon56}.

The graph approach was also used in the investigation of quantum non-locality,   one of the distinctive features of quantum theory.  This describes the correlations of local measurement
on separate subsystems. It has been proved that non-locality is a fundamental resource in a variety of practical applications, ranging from quantum key distribution \cite{Acin2007} to quantum communication complexity \cite{Buhrman2010}.

Another phenomenon called quantum contextuality had been proposed earlier, and states that outcomes cannot be assigned to measurements independently of the
contexts of the measurements. This phenomenon is known as the Kochen-Specker  paradox ~\cite{Kochen1967}.  One critical observation about contextuality and non-locality is that the non-locality is a special case of contextuality: the compatibility of the measurement outcomes are given by the measurement of observables on separable subsystems, and contextuality holds even for single systems.
Besides, contextuality has been proved to be useful in quantum computation \cite{Anders2009,Raussendorf2013,Howard2014,Delfosse2015,Bermejo2017},
which has attracted lots of attention to this topic in recent years.
Several  approaches
 have been proposed in the investigation of   the phenomenon of contextuality \cite{Abramsky2011,Cabello2014}. A  very interesting combinatoric method uses hypergraphs to describe contextual scenarios~\cite{Acin15};  the authors apply graph invariants to the investigation of the classification of  probabilistic models.

\subsection{Main Results}
We need a few basic concepts in order to state our main conclusions.  
\begin{Def}\label{space o}
Let $\real^n_+$  denote vectors in $\real^n$, all of whose coordinates  are nonnegative.
Given a function $f: \real^n_+\to \real_+$, 
\begin{itemize}
\item[{\rm(1)}] $f$ is  positive affine, if   $f(\lambda p)=\lambda f(p)$,  $\forall \lambda >0$, $\forall p\in \real^n_+$.
\item[{\rm(2)}] $f$ is nondegenerate, if  $\forall p\in \real^n_+$ with $p \neq 0$, then $f(p)>0$.
\item[{\rm(3)}] $f$ is  bounded, if there exist constants $c_1, c_2>0$ such that
\begin{eqnarray}
c_1\norm{p} \leq f(p)\leq c_2\norm{p},~\forall p \in \real^n_+\;.
\end{eqnarray}
\item[{\rm(4)}] $f$ is  continuous if it is pointwise continuous.
\item[{\rm(5)}] $f$ is  convex, if  for any $p_1, p_2\in\real^n_+$, $ \mu\in [0, 1]$,
\begin{eqnarray}
f(\mu p_1+(1-\mu)p_2)\leq \mu f(p_1)+(1-\mu)f(p_2)\;.
\end{eqnarray}
 \item[{\rm(6)}]  $f$ is monotonically increasing, if $f(p)\leq f(p+q)$ for any $p, q\in \real^n_+$.
\end{itemize}
\end{Def}

\begin{Def} \label{DefSpace-Trans}
Denote by $\mathcal{O}$ the set of functions
$f: \real^n_+ \to \real_+$ being: {\rm(1)}~positive affine, {\rm(2)}~non-degenerate,  {\rm(3)}~bounded, and  {\rm(4)}~continuous.  Denote by $\mathfrak{B}$  the transformation with domain $\mathcal{O}$,
 \begin{eqnarray}
(\mathfrak{B}f)(p)=\sup_{q\neq 0}\frac{\inner{p}{q}}{f(q)}\;.
\end{eqnarray}
\end{Def}

\begin{thm}[\bf  Involution] \label{thm:final}
The transform $\mathfrak{B}$  maps $\mathcal{O}$ to itself.  
%, $\mathfrak{B}\mathcal{O}\subset \mathcal{O}$.  
 Given $f\in\mathcal{O}$, one has $\mathfrak{B}^{2}(f)=f$,
  %satisfies  on  $f\in\mathcal{O}$,
 iff $f$
 is monotonically increasing and convex.
% :
%$$f(p)\leq f(p+q)\leq f(p)+f(q)\;,
%\quad \text{for all }p, q\in \mathbb{R}^n_+\;.$$
\end{thm}

Consider a graph $G=(V, E)$  with vertices $V=V(G)$ and edges $E=E(G)$.  Let  $\overline{G}$ denote the complement of $G$.  The vertices  of $G$ and $\overline{G}$ are the same, but the edges are complimentary.  Two vertices in $G$  define an  edge of $\overline{G}$,  iff they do not define an edge in $G$.

A weight on $G$ is a function $p:V(G)\rightarrow [0, +\infty)$.
 In \S\ref{sec:pre} we discuss weighted  invariants for a graph $G$, that depend  on a weight (function) $p$. We define  the weighted independence number $\alpha_{G}(p)$, the weighted Lov\'{a}sz number  $\vartheta_{G}(p)$,  the weighted fractional packing number $\alpha_{G}^{*}(p)$, and the weighted Shannon capacity $\Theta_{G}(p)$.   
% Let $\mathfrak{B}$ act on the weight as
% 	\be
%		(\mathfrak{B}\alpha_{G})(p) = \sup_{q\neq0} \frac{\lra{p,q}}{\alpha_{G}(q)}\quad\text{and}\quad
%		(\mathfrak{B}\vartheta_{G})(p) = \sup_{q\neq0} \frac{\lra{p,q}}{\vartheta_{G}(q)}
%		\;.
%	\ee

\begin{thm}[\bf Duality]\label{thm:g_dual}
Given a graph $G$ and a weight function $p$,
\begin{equation}
\label{eq:dua}\alpha_{G}=\mathfrak{B}\alpha_{\overline{G}}^*\;,\quad
\alpha^*_{G}=\mathfrak{B}\alpha_{\overline{G}}\;,\quad
\vartheta_{G}=\mathfrak{B}\vartheta_{\overline{G}}\;.
\end{equation}
\end{thm}

Since the weighted independence number, weighted
  Lov\'{a}sz number and weighted fractional packing number, satisfy the conditions of Theorem~\ref{thm:final}, they equal their double transform. On the other hand the weighted Shannon capacity $\Theta_{G}$ is not convex \cite{Acin15}. So we infer:

\begin{cor}\label{cor:dual}
Given a graph $G$ and a weight function $p$,
\begin{eqnarray}
\label{eq:ddua}&&\alpha_{G}=\mathfrak{B}^{2}\alpha_{{G}}\;,\quad
\alpha^{*}_{G}=\mathfrak{B}^{2}\alpha_{G}^{*}\;,\quad
\vartheta_{G}=\mathfrak{B}^{2}\vartheta_{G}\;. 
\end{eqnarray}
In  case that  $\Theta_{G}$ is not convex, $\Theta_{G} \neq \mathfrak{B}^{2} \Theta_{G}$.
\end{cor}

We also consider the application of the duality relation of graph invariants in contextuality. In Theorem~\ref{thm:space} we find that some graph invariants used in the area of contextuality can be interpreted as the max-relative entropy of the probabilistic models.  The max-relative entropy has proved to play an important role in quantum communication~\cite{KRS09}, quantum entanglement manipulation~\cite{Dat09a,Dat09b,BD11}, and quantum coherence manipulation~\cite{Bu17}.  So it may be possible to use graph invariance to provide additional understanding of the role of contextuality in quantum information tasks.
In addition in Theorem \ref{thm:phy} we show that the graph invariant $\alpha^*$ gives a lower bound on the maximum violation of all positive
Bell-type inequalities of the probabilistic models.

\section{Preliminaries}\label{sec:pre}
In this section, we introduce some graph invariants that we study in this paper.  
%\subsection{Graph invariants}\label{sect:DefInvariant}
Consider an undirected and loopless graph $G=(V, E)$, the $V$ the set of vertices and $E$  the set of edges~\cite{Berge1973}. If the vertices $v_1, ~v_2$ are in the same edge, we say that they are adjacent, denoted $v_1 \sim v_2$; otherwise they are non-adjacent $v_1 \nsim v_2$.

\begin{Def} Given $G=(V,E)$,  $I\subset V$ is independent, if for all distinct $v_1,~v_2\in I$, $v_1\nsim v_2$.  Let $\mathcal{I}(G)$ denote the set of independent subsets of  $V$.
The independence number $\alpha_{G}$ of a graph $G$, is 
\[
\alpha_{G}:=\max_{I\in\mathcal{I}(G)} \vert I\vert\;.
\]
Given a weight $p$ on $G$, the weighted independence  number $\alpha_G( p)$ of $G$, is
\begin{equation*}
 \alpha_G(p):=\max_{I\in\mathcal{I}(G)}\sum_{v\in I}p(v)\;.
 \end{equation*}
\end{Def}

Let $G_1$ and $G_2$ be two graphs. Their strong product~\cite{Imrich2000},  denoted $G_1\boxtimes G_2$, is the graph with vertices
$V(G_1\boxtimes G_2)=V(G_1)\times V(G_2)$.  The edges satisfy
$(u_1, v_1)\sim (u_2, v_2)$ if and only if
\[
(u_1\sim u_2\wedge v_1\sim v_2)\vee (u_1\sim u_2\wedge v_1=v_2) \vee (u_1=u_2\wedge v_1\sim v_2)\;.
\]

\begin{Def}[\cite{Shannon56}]
The  Shannon capacity of $G$ is
\[\Theta_{G}=\lim_{n\rightarrow \infty}\sqrt[n]{\alpha_{G^{\boxtimes n}}}\;.\]
Given a weight on $G$, the weighted Shannon capacity is
\[\Theta_G( p)=\lim\limits_{n\rightarrow \infty}\sqrt[n]{\alpha_{G^{\boxtimes n}}(p^{\otimes n})}\;.\]
\end{Def}

To define the Lov\'{a}sz number of a graph, we first need an additional property.  If $|V(G)|=n$, an orthonormal representation of $G$ \cite{Lovasz1979} is an assignment
\[
\psi:V(G)\rightarrow\mathbb{R}^n\quad
\text{with}\quad
v\mapsto \psi_v\;,
\]
such that: ~
(1) each $\psi_v$ is an unit vector;~~
(2) $u \nsim v $ implies $\psi_u ~\bot~ \psi_v$.

\begin{Def}[\cite{Lovasz1979}]
The Lov\'{a}sz number, $\vartheta_{G}$, is defined as:
\begin{eqnarray*}
\vartheta_{G}=\max_{c, \psi}
\sum_{v\in V(G)}
|\langle c , \psi_v \rangle|^2,
\end{eqnarray*}
where $c \in\mathbb{R}^n$ ranges over all unit vectors and $\psi$ over all orthonormal representations of $\overline{G}$. The weighted Lov\'{a}sz number  is
\begin{eqnarray*}
\vartheta_G(p)=\max_{c, \psi}
\sum_{v\in V(G)}
p(v)|\langle c , \psi_v \rangle|^2,
\end{eqnarray*}
where $c \in\mathbb{R}^n$ ranges over all unit vectors and $\psi$ over all orthonormal representations of  $\overline{G}$.
\end{Def}

\begin{Def}
A subset $C$ of $V(G)$ is called a clique, if for any $v_1,~v_2\in C$, $v_1\sim v_2$. The fractional packing number $\alpha^*_{G}$ of $G$ is:
\begin{eqnarray*}
\alpha^*_{G}=\max_q
\sum_{v\in V(G)} q(v),
\end{eqnarray*}
~~~~where $q : V(G) \rightarrow [0, 1]$ ranges over all the vertex weightings satisfying
$\sum\limits_{v\in C}q(v)\leq 1$ for all cliques $C$ of $G$. The weighted fractional packing number $\alpha^*_G( p)$ is:
\begin{eqnarray*}
\alpha^*_{G}(p)=\max_q
\sum_v p(v)q(v),
\end{eqnarray*}
where $q$ is taken as above.
\end{Def}

\section{Proof of the Duality Theorem \ref{thm:g_dual}}\label{sect:DualityProof}
%between graph invariants }
%In this section, we prove Theorem \ref{thm:g_dual}, which gives a dual relation between  weighted graph invariants.
%
%\begin{proof}[Proof of Theorem \ref{thm:g_dual}]
First we prove that  $\alpha^*_{G}=\mathfrak{B}\alpha_{\overline{G}}$, or equivalently $\alpha^*_{\overline{G}}=\mathfrak{B}\alpha_{{G}}$.  We omit the identical proof that  $\alpha_{G}=\mathfrak{B}\alpha_{\overline{G}}^*$. 
% and $\alpha^*_{G}=\mathfrak{B}\alpha_{\overline{G}}$ are almost the same,
Assume that we have a nonzero weight $w$ on $G$. It is easy to see that  the vertices of a clique in $\overline{G}$ also form an independent set
of $G$.  According to the definition of
$\alpha_G(w)$, the weight function $q(v)={w(v)}/{\alpha_G( w)}$
%\begin{eqnarray*}
%q(v)=\frac{w(v)}{\alpha_G( w)}
%\end{eqnarray*}
 satisfies the condition that
$\sum\limits_{v\in C}q(v)\leq 1$ for any clique $C$ in $\overline{G}$.
From the definition of $\alpha^*_{\overline{G}}(p)$, we infer that for any $w\neq0$, 
\[
\alpha^*_{\overline{G}}(p)\geq \sum\limits_v p(v)q(v)=\frac{\inner{p}{w}}{\alpha_G(w)}\;.
\]
Since the set
\begin{eqnarray*}
\set{q:V(G)\rightarrow \real_+ \left | \ \sum_{v\in C}q_v\leq 1\;,\text{ for any clique C in }\overline{G} \right.}
\end{eqnarray*}
is compact, then there exists some weight function $q_0$ in it such that  
$\alpha^*_{\overline{G}}(p)=\sum_v p(v)q_0(v)$.
It is easy to see that $\alpha_G(q_0)\leq 1$, thus
$\alpha^*_{\overline{G}}(p)\leq {\inner{p}{q_0}}/{\alpha_G(q_0)}$.
Hence, we have proved the equation $\alpha^*_{G}=\mathfrak{B}\alpha_{\overline{G}}$.

To prove $\vartheta_{G}=\mathfrak{B}\vartheta_{\overline{G}}$,
we only need to prove
$%\begin{eqnarray}
\vartheta_{\overline{G}}(w)=\max_{p\neq 0}{\inner{p}{w}}/{\vartheta_G( p)}
$. %\end{eqnarray}
From the definition of the weighted Lov\'asz number, we know that
\begin{eqnarray*}
\vartheta_{\overline{G}}(w)=
\max_{\ket{\Psi}, \psi}\sum_{u\in V(G)}w(u)|\langle \Psi | \psi_u \rangle|^2,
\end{eqnarray*}
where $\ket{\Psi}\in\mathbb{R}^{|V(G)|}$ ranges over all unit vectors and
$\psi$ over all orthonormal
representations of G.
It has been proved in \cite{Acin15} that the weighted Lov\'{a}sz number can also be written as
\begin{eqnarray*}
\vartheta_G(p)=\min_{\ket{\Psi}, \psi}
\max_{v\in V}
\frac{p_v}{|\langle \Psi | \psi_v \rangle|^2},
\end{eqnarray*}
where $\ket{\Psi}\in\mathbb{R}^{|V(G)|}$ ranges over all unit vectors and
$\psi$ over all orthonormal
representations of G.
Therefore, there exists some unit vector and an orthogonal representation $(\ket{\Psi}, \psi)$
such that
\begin{eqnarray*}
\vartheta_G(p)=
\max_{v\in V}\frac{p_v}{|\langle \Psi | \psi_v \rangle|^2}\;,
\end{eqnarray*}
so
\begin{eqnarray*}
\vartheta_G(p)\vartheta_{\overline{G}}(w)
&\geq&\max_{v \in V}\frac{p_v}{|\langle \Psi | \psi_v \rangle|^2}
\sum_{u\in V}w(u)|\langle \Psi | \psi_u \rangle|^2\\
&\geq& \sum_{u \in V}p(u)w(u)\;,
\end{eqnarray*}
that is
\begin{eqnarray}
\vartheta_{\overline{G}}(w)\geq\max_{p\neq 0}\frac{\inner{p}{w}}{\vartheta_G( p)}\;.
\end{eqnarray}
Besides, there exists some
unit vector and a orthogonal representation  $(\ket{\Psi}, \psi)$
such that
\begin{eqnarray*}
\vartheta_{\overline{G}}(w)=\sum_{u\in V}w(u)|\langle \Psi | \psi_u \rangle|^2.
\end{eqnarray*}
Now let us take the weight function
p such that
$\frac{p_u}{|\langle \Psi | \psi_u \rangle|^2}=\frac{p_v}{|\langle \Psi | \psi_v \rangle|^2}$
for any $u,v$, then
\begin{eqnarray*}
 \sum_{u \in V}p(u)w(u)
&=&\max_{v \in V}\frac{p_v}{|\langle \Psi | \psi_v \rangle|^2}
\sum_{u\in V}w(u)|\langle \Psi | \psi_u \rangle|^2\\
&\geq& \vartheta_G(p)\vartheta_{\overline{G}}(w)\;,
\end{eqnarray*}
Therefore, we get the desired result and the theorem is proved.
%\begin{eqnarray*}
%\vartheta_{\overline{G}}(w)=\max_{p\neq 0}\frac{\inner{p}{w}}{\vartheta_G( p)}.
%\end{eqnarray*}
%
%\end{proof}

\section{Mathematical Properties of the Transformation $\mathfrak{B}$}\label{sec:math_prop}
All the statements in Corollary~\ref{cor:dual} follow from Theorem \ref{thm:g_dual}, except that we need to show that the Shannon capacity may not be invariant under $\mathfrak{B}^2$. To exhibit this, we establish necessary and sufficient conditions for functions to be $\mathfrak{B}^2$ invariant. That is the focus of this section, in which we investigate the mathematical properties of $\mathfrak{B}$ in detail.  
%to
%obtain the sufficient and necessary conditions for the functions to be $\mathfrak{B}^2$ invariant or $\mathfrak{B}$ invariant.   
Recall the space $\mathcal{O}$ given in Definitions \ref{space o} and \ref{DefSpace-Trans}.

%\begin{Def}\label{space o}
%For a function $f: \real^n_+\to \real_+$, say
%\begin{itemize}
%\item[{\rm(1)}] $f$ is  positive affine, if   $f(\lambda p)=\lambda f(p)$,  $\forall \lambda >0$, $\forall p\in \real^n_+$;
%\item[{\rm(2)}] $f$ is nondegenerate, if  $\forall p\in \real^n_+$ with $p \neq 0$, then $f(p)>0$;
%\item[{\rm(3)}] $f$ is  bounded, if there exist constants $c_1, c_2>0$ such that
%\begin{eqnarray}
%c_1\norm{p}_2\leq f(p)\leq c_2\norm{p}_2,~\forall p \in \real^n_+;
%\end{eqnarray}
%\item[{\rm(4)}] $f$ is  continuous if it is pointwise continuous;
%\item[{\rm(5)}] $f$ is  convex, if  for any $p_1, p_2\in\real^n_+$, $ \mu\in [0, 1]$,
%\begin{eqnarray}
%f(\mu p_1+(1-\mu)p_2)\leq \mu f(p_1)+(1-\mu)f(p_2)\;;
%\end{eqnarray}
% \item[{\rm(6)}]  $f$ is monotonically increasing, if $f(p)\leq f(p+q)$ for any $p, q\in \real^n_+$.
%\end{itemize}
%\end{Def}

%Note that the weighted graph invariants $\alpha_G, \Theta_G, \vartheta_G, \alpha^*_G$
%satisfy the conditions (1)--(4).  This motivates our introduction of $\mathcal{O}$ in Definition~\ref{DefSpace-Trans}.

\begin{prop}\label{prop:space}
If $f\in \mathcal{O}$ , then $\mathfrak{B}f\in \mathcal{O}$.
\end{prop}

\begin{proof}
That $\mathfrak{B}f\geqslant 0$, is clear from the definition of $\mathfrak{B}$.
We now verify properties (1--4) for $\mathfrak{B}f$.

%\noindent
(1) {\em Positive affine}: Let $\lambda>0$. Then
\[
(\mathfrak{B}f)(\lambda p)=\sup\limits_{q\neq 0}\frac{\inner{\lambda p}{q}}{f(q)}=\lambda\sup\limits_{q\neq 0}\frac{\inner{ p}{q}}{f(q)}=\lambda (\mathfrak{B}f)(p)\;.
\]

(2) {\em Nondegenerate}: 
For $p\neq 0$, we have $(\mathfrak{B}f)(p)\geq \frac{\inner{p}{p}}{f(p)}>0$.

(3) {\em Bounded}: Since $c_1\norm{p}\leq f(p)\leq c_2\norm{p}$,  one has 
\[
c_2^{-1}\norm{p}\leq (\mathfrak{B}f)(p)\leq c_1^{-1}\norm{p}\;.
\]

(4) {\em Continuity}: For any $p, p'\in\real^n_+$,
\[
|(\mathfrak{B}f)(p)-(\mathfrak{B}f)(p')|\leq \sup_{q\neq 0} |\frac{\inner{p-p'}{q}}{f(q)}|
\leq c_1^{-1}\norm{p-p'}\;.
\]
\end{proof}

We say $f\leq g$ if  $f(p)\leq g(p)$ for every $p\in\real^n_+$.
\begin{prop}\label{prop:ddual}  Let $f, g\in \mathcal{O}$.   
%\begin{itemize}
%\item[{\rm(1)}] 
If $f\leq g$, then $\mathfrak{B}f\geq \mathfrak{B}g $. Also 
%\item[{\rm(2)}] If $f\in \mathcal{O}$ , then 
$\mathfrak{B}^2f\leq f$,
and $\mathfrak{B}^3f=\mathfrak{B}f$.
%\end{itemize}
\end{prop}

\begin{proof}
The fact that $\mathfrak{B}f\geq \mathfrak{B}g $ follows from the definition of $\mathfrak{B}$.  Also  
 \[
 (\mathfrak{B}^2f)(p)
 =\sup_{q\neq 0} \frac{\inner{p}{q}}{(\mathfrak{B}f)(q)}
 =\sup_{q\neq 0}\frac{\inner{p}{q}}{\sup\limits_{r\neq0}\frac{\inner{r}{q}}{f(r)}}\leq \sup_{q\neq 0}\frac{\inner{p}{q}}{\frac{\inner{p}{q}}{f(p)}}=f(p)\;.
 \]
 Since  $\mathfrak{B}^3f\leq\mathfrak{B}f$ and $\mathfrak{B}^2f\leq f$, we infer that $\mathfrak{B}^3f\geq\mathfrak{B}f$.
\end{proof}
%
%
%\begin{cor}
%For any $f\in \mathcal{O}$, $\mathfrak{B}^3(f)=\mathfrak{B}(f)$.
%\end{cor}
%
%\begin{proof}By Proposition \ref{prop:ddual}(2), $\mathfrak{B}^3(f)\leq\mathfrak{B}(f)$ and $\mathfrak{B}^2(f)\leq f$; so by Proposition \ref{prop:ddual}(1), $\mathfrak{B}^3(f)\geq\mathfrak{B}(f)$.
%\end{proof}

Now, let us focus on the set 
\[
\Delta:=\set{f\in \mathcal{O}| \ \mathfrak{B}^2f=f}\;.
\]
Note that $\alpha_G, \vartheta_G, \alpha^*_G\in \Delta$,
according to Theorem~\ref{thm:g_dual}.

%\begin{prop}
%The function $\mathfrak{B}^2(f)$ is maximal in $\Delta$ in the sense that 
%{\rm(1)} Both $\mathfrak{B}^2(f)\in \Delta$, and $\mathfrak{B}^2(f)\leq f$, as well as 
%
%{\rm(2)} For any $g\in \Delta$ with $g\leq f$, it is true that $g\leq \mathfrak{B}^2(f)$.
%\end{prop}

\begin{prop}
For $f\in \mathcal{O}$, $\mathfrak{B}^2f$ can be defined by the following universal properties:

{\rm(1)} Both $\mathfrak{B}^2f\in \Delta$, and $\mathfrak{B}^2f\leq f$, as well as 

{\rm(2)} For any $g\in \Delta$ with $g\leq f$, it is true that $g\leq \mathfrak{B}^2f$.
\end{prop}

\begin{proof}
From Proposition \ref{prop:ddual}, we infer that  (1) holds. Since $g\leq f$ and $\mathfrak{B}g \geq \mathfrak{B}f$, it follows that $g=\mathfrak{B}^2g \leq \mathfrak{B}^2f$.
\end{proof}

Now we prove Theorem \ref{thm:final} by steps.

\begin{Def}
Let $ f\in \mathcal{O}, \lambda>0$, denote
\begin{eqnarray}
C_f(\lambda)&:=&\set{p\in\real^n_+| \ f(p)<\lambda}\;. \\
\partial C_f(\lambda)&:=& \set{p\in\real^n_+| \ f(p)=\lambda}\;.
\end{eqnarray}
\end{Def}

\begin{lem}\label{lem: convex}
Let $f\in \mathcal{O}$. The function $f$ is convex, iff  $C_f(\lambda)$ is convex for every $\lambda >0$.
\end{lem}
\begin{proof}
If the function $f$ is convex, then convexity of  $C_f(\lambda)$ is a consequence of the convexity of $f$.
In the converse direction, we need only prove that for every $p_1, p_2\in \real^n_+$, and for every $\mu\in [0,1]$,
\begin{eqnarray*}
f(\mu p_1+(1-\mu) p_2)\leq \mu f(p_1)+(1-\mu) f(p_2)\;.
\end{eqnarray*}

 If either $p_1$ or $p_2$ equals $0$, e.g.,
$p_2=0$, then $f(\mu p_1+(1-\mu) p_2)= \mu f(p_1)$.
Otherwise, both $p_1, p_2 \neq 0$, and there exists
$\alpha>0$ such that  $f(\alpha p_2)=f(p_1)$.  Moreover, there exists  $\lambda>0$  such that
$f(p_1)<\lambda$.  Therefore,
\begin{eqnarray*}
&&\hskip-.5inf(\mu p_1+(1-\mu)p_2)
=f\left(\mu p_1+\frac{(1-\mu)}{\alpha}(\alpha p_2)\right)\\
&=&\left(\mu +\frac{(1-\mu)}{\alpha}\right)f\left( \frac{\mu}{\mu+(1-\mu)/\alpha} p_1+\frac{(1-\mu)/\alpha}{\mu+(1-\mu)/\alpha}\alpha p_2\right)\;.
\end{eqnarray*}
Since $p_1, \alpha p_2\in C_f(\lambda)$, then
$\frac{\mu}{\mu+(1-\mu)/\alpha} p_1+\frac{(1-\mu)/\alpha}{\mu+(1-\mu)/\alpha}\alpha p_2\in C_f(\lambda)$.
That is
\begin{eqnarray*}
 f(\mu p_1+(1-\mu)p_2)<\left(\mu +\frac{(1-\mu)}{\alpha}\right) \lambda\;.
 \end{eqnarray*}
Since we can choose  $\lambda$ arbitrarily,
\begin{eqnarray*}
f(\mu p_1+(1-\mu)p_2)&\leq& \left(\mu +\frac{(1-\mu)}{\alpha}\right)f(p_1)\\
&=&\mu f(p_1)+(1-\mu)f(p_2)\;,
\end{eqnarray*}
completing the proof.
\end{proof}

\begin{lem}
Let $f\in \mathcal{O}$, and $\lambda>0$.  Then $C_f(\lambda)$ is connected.
\end{lem}

\begin{proof}
If $p \in C_f(\lambda)$, then for any $0<\mu < 1$, $\mu p \in C_f(\lambda)$. Thus, $p$ and $0$ are path-connected, and $C_f(\lambda)$ is connected.
\end{proof}

\begin{thm}\label{12}
Let $f\in \mathcal{O}$.  Then
$f\in \Delta $  iff  for any $p\in\real^n_+$ with $ p\neq 0$,  there exists a hyperplane $H=\set{r\in \real^n_+| \inner{r}{q_0}=\inner{p}{q_0}}$
such that $f(p)=\inf\limits_{r\in H}f(r)$ and $q_0\in \real^n_+$.

\end{thm}

\begin{proof}
Since $\mathfrak{B}^2f\leq f$, then $f\in \Delta$ iff $\mathfrak{B}^2f\geq f$.  If $\mathfrak{B}^2f\geq f$, then for any $p\in \real^n_+$,
\begin{eqnarray*}
(\mathfrak{B}^2f)(r)=\sup_{q\in \mathbb{R}^n_+, q\neq 0}\inf_{r\neq 0}
\frac{\inner{p}{q}}{\inner{q}{r}}f(r)\geq f(p)\;.
\end{eqnarray*}
Thus, there exists $q_0\in \real^n_+$ such that $\inf\limits_{r\neq 0 } \frac{\inner{p}{q_0}}{\inner{q_0}{r}}f(r)\geq f(p)$,
which leads to
 \begin{eqnarray*}
 \inf_{\inner{r}{q_0}=\inner{p}{q_0}}f(r)\geq f(p)\;.
 \end{eqnarray*}
Let  $H=\set{r\in \real^n_+| \inner{r}{q_0}=\inner{p}{q_0}}$ , then $p\in H$ and $f(p)=\inf\limits_{q\in H}f(q)$ .

On the other hand,
if there exists a hyperplane 
	\[
	H=\set{r\in \real^n_+| \inner{r}{q_0}=\inner{p}{q_0}}
	\]  
such that  $f(p)=\inf\limits_{r\in H}f(r)$ and $q_0\in \real^n_+$, then
 \begin{eqnarray*}
 (\mathfrak{B}^2f)(p)
 =\sup\limits_{q \in \mathbb{R}^n_+ q\neq 0}\inf\limits_{r'\neq 0}\frac{\inner{p}{q}}{\inner{q}{r'}}f(r') \geq\inf\limits_{r'\neq 0} \frac{\inner{p}{q_0}}{\inner{q_0}{r'}}f(r')= \inf\limits_{r'\neq 0}f\left( \frac{\inner{p}{q_0}}{\inner{q_0}{r'}}r'\right)\geq f(p)\;.
 \end{eqnarray*}
The last inequality comes from the fact that $\frac{\inner{p}{q_0}}{\inner{q_0}{r'}}r' \in H$.
\end{proof}

We infer from Theorem \ref{12} that  for any $f\in \Delta$ and  $p\in \real^n_+$,  there exists a vector $q_0\in \real^n_+$ such that the hyperplane $H=\set{r\in \real^n_+| \inner{r}{q_0}=\inner{p}{q_0}}$
contains the point $p$, and $f(p)=\min\limits_{r\in H}f(r)$. We denote this hyperplane by $H_p$.
And we also denote $H^+_p$ and $H^{-}_p$ as follows:
\begin{eqnarray*}
H^+_p&:=\set{r\in \real^n_+| \inner{r}{q_0}>\inner{p}{q_0}}\;,\quad
H^-_p&:=\set{r\in \real^n_+| \inner{r}{q_0}<\inner{p}{q_0}}\;.
\end{eqnarray*}

\begin{prop}\label{13}
If $f\in \Delta $ and $0 \neq p\in\real^n_+$, then
$C_f(f(p))\subset H^-_p$.
\end{prop}
\begin{proof}
It is easy to see that $0\in C_f(f(p))\cap H^-_p$.
Besides,
since $f(p)=\inf\limits_{q\in H_p}f(q)$, then $H_p\cap C_f(f(p))= \emptyset$. Also $C_f(f(p))$ is connected, so $H_p^+\cap C_f(f(p))= \emptyset$. Therefore  $C_f(f(p))\subset H^-_p$.
\end{proof}

\begin{prop}\label{prop:C_f}\textsf{}
If $f\in \Delta$, then
\begin{eqnarray}
C_f(\lambda)= \real^n_+ \cap \bigcap\limits_{p\in \partial C_f(\lambda)}H_p^-\; ,\; \forall \lambda>0\;.
\end{eqnarray}
\end{prop}
\begin{proof}
From Proposition \ref{13},  we infer that  $ C_f(\lambda)\subseteq H_p^-$ for every $\lambda>0$. Thus $C_f(\lambda)\subseteq \real^n_+ \cap \bigcap\limits_{p\in \partial C_f(\lambda)}H_p^-$.
On the other hand, if there exists a vector $q\in \real^n_+ \cap \bigcap\limits_{p\in \partial C_f(\lambda)}H_p^-$ such that $q\notin C_f(\lambda)$, then $\exists~ 0< \alpha \leq 1$ such that $f(\alpha q)=\lambda$, i.e.,  $\alpha q\in \partial C_f(\lambda)$. Since $\alpha \leq 1$, then $q\notin H^-_{\alpha q}$, which contradicts to the choice of $q\in \real^n_+ \cap \bigcap\limits_{p\in \partial C_f(\lambda)}H_p^-$.
\end{proof}

\begin{cor} Let  $f\in \Delta$, then $f$ is convex.

\end{cor}
\begin{proof}
This corollary comes directly from Lemma \ref{lem: convex}, Proposition \ref{prop:C_f}, and the fact
$H^{-}_p$ is convex.
\end{proof}

\begin{Def} A convex subset $C\subset \real^n_{+}$ is positive convex, if
 for any $p_0\notin C$, there exists $ q\in \real^n_{+}$, and $\gamma>0$, such that both $\inner{p_0}{q}\geqslant \gamma$, and also $\inner{p}{q}<\gamma$ for all $p\in C$.
\end{Def}

\begin{thm}\label{1.1} If $f \in \mathcal{O}$, then $f\in \Delta $ iff $C_f(\lambda)$ is positive convex for evrery $\lambda>0$.
\end{thm}
\begin{proof}"$\Rightarrow$":
 We have proved that $C_f(\lambda)$ is convex. $\forall p_0 \notin C_f(\lambda)$, i.e. $f(p_0) \geq \lambda$,
 let $\alpha >0$ such that $f(\alpha p_0)=\lambda$, so $\alpha \leq 1$. By Theorem \ref{12}, $\exists q\in \real^n_{+}$ such that on the hyperplane $H_{\alpha p_0}=\set{r\in \real^n_{+}\big| \inner{r}{q}=\inner{\alpha p_0}{q}}$ we have $\lambda = f(\alpha p_0)= \inf\limits_{r\in H_{\alpha p_0}}f(r)$ and $C_f(\lambda)\subseteq H_{\alpha p_0}^-$, where $C_f(\lambda)\subseteq H_{\alpha p_0}^-$
 means that $\inner{p}{q}<\inner{\alpha p_0}{ q}$  for every  $p\in C_f(\lambda)$.
 While $\inner{p_0}{q}\geq \inner{\alpha p_0}{q}$, therefore $C_f(\lambda)$ is a positive convex set.

"$\Leftarrow$": $\forall p_0 \in \real^n_{+}$, let $\lambda = f(p_0)$. Because $C_f(\lambda)$ is positive convex, there exists $q_0\in \real^n_{+},~\gamma >0$ such that $\inner{p_0}{q_0}= \gamma$ and $\forall p \in C_f(\lambda)$, $\inner{p}{q_0}<\gamma$. Let $H=\set{r\in \real^n_{+}\big| \inner{r}{q_0}= \gamma}$. So $\forall r \in H,~f(r) \geq \lambda = f(p_0)$, that is $f(p_0)=\min\limits_{r\in H}f(r)$. By Theorem \ref{12}, $f\in \Delta$.
\end{proof}

\begin{prop} \label{lem:mon}
If  $f\in \mathcal{O}$, then f is  monotonically increasing and convex if and only if $C_f(\lambda)$ is positive convex for every $ \lambda>0$.
\end{prop}
\begin{proof}"$\Rightarrow$":
We have shown that
 $C_f(\lambda)$ is convex. To prove it is positive convex, let us take any vector $p_0=(\alpha(1),\alpha(2),...,\alpha(n))\in \real^n_{+}$, with $p_0 \notin C_f(\lambda)$.  There exists some
$0<\beta \leq 1$ such that $f(\beta p_0)=\lambda$, so $\beta p_0 \notin C_f(\lambda)$. Since $C_f(\lambda)$ is convex and open, there must be a vector $q=(q(1),...,q(n))$ such that $\inner{\beta p_0}{q}=\gamma$ and  $\inner{p}{q}<\gamma$ for any $ p\in C_f(\lambda)$.

\noindent {\bf Case (1)} Suppose   $\alpha(1)\cdots\alpha(n)>0$.  In this case, if $q\notin \real^n_{+}$, we assume $q_i<0$. Assume the vector $p'=\beta p_0-(0,0,...,0,\beta\alpha(i),0,...,0)$. We have $\inner{p'}{q}>\inner{\beta p_0}{q}=\gamma$. Thus there exists $0<\theta<1$ such that $\inner{\theta p'}{q}=\gamma$. By monotonicity, $f(\theta p')<f(p')\leqslant f(\beta p_0)=\lambda$, which contradicts to the choice of $q$. Thus  $q\in \real^n_{+}$, and $\inner{p_0}{q}\geq\inner{\beta p_0}{q}=\gamma$.

\noindent{\bf Case (2)} If  $\alpha(1)\cdots \alpha(n)=0$,  let $p_k':=(\beta\alpha(1)+\frac{1}{k},...,\beta\alpha(n)+\frac{1}{k})$, and  $p_k=\theta_k p_k'$ such that $f(p_k)=\lambda$ with
 $0<\theta_k\leq1$. Then the vectors $\{p_k\}$ satisfy: 
\[
\text{(a) } p_k' \in \real^n_{\geq 0}\;,
\quad
\text{(b) } f(p_k)=\lambda, \forall k\;, \quad
\text{(c) }\lim\limits_{k\rightarrow\infty}p_k=\beta p_0\;.
\]
 As in the proof of Case (1), for each $p_k$,
 there exists a $q_k\in \real^n_{+}$ such that $\|q_k\|=1$ and $\inner{p}{q_k}<\inner{p_k}{q_k}$ for any   $p\in C_f(\lambda)$. Since the unit ball of a finite dimensional space is compact, we can choose a subsequence $\{q_{k_t}\}$ that converges to a vector $Q$, i.e.,   $Q\in \real^n_{+}$, and $\inner{\beta p_0}{Q}= \lim\limits_{t\rightarrow \infty}\inner{p_{k_t}}{q_{k_t}}$. Therefore,  $\inner{p}{Q}=\lim\limits_{t\rightarrow \infty}\inner{p}{q_{k_t}}\leq \inner{\beta p_0}{Q}$ for any $p\in C_f(\lambda)$. Since $C_f(\lambda)$ is open, then
 $\inner{p}{Q}<\inner{\beta p_0}{Q}\leq \inner{p_0}{Q}$, which means $C_f(\lambda)$
 is positive convex.

"$\Leftarrow$": According to Lemma \ref{lem: convex}, f is a convex function.
  For any $ p_0,q \in \real^n_{+}$, let $\lambda=f(p_0)$, then $p_0\in C_f(\lambda)$. Since $C_f(\lambda)$ is positive convex,
 there exist $q_0\in \real^n_{+}$ and $\gamma>0$ such that $\inner{p_0}{q_0}\geq \gamma$, and $\inner{p}{q_0}<\gamma$ for any  $ p\in C_f(\lambda)$. Besides, $\inner{p_0+q}{q_0}\geq \inner{p_0}{q_0}\geq \gamma$, hence
 $p_0+q\notin C_f(\lambda)$. Thus we have $f(p_0+q)\geq \lambda=f(p_0)$. Therefore f is monotonically increasing.
\end{proof}

\begin{proof}[Proof of Theorem \ref{thm:final}]
This is a direct consequence of Theorem \ref{1.1} and Proposition \ref{lem:mon}.
\end{proof}

Since   the weighted graph invariants $\alpha_G( p), \Theta_G(p), \vartheta_G(p), \alpha^*_G(p)$ all
belong to $\mathcal{O}$, the above theorem gives us a sufficient and necessary condition for them to be $\mathfrak{B}^2$ invariant.
Moreover, a direct corollary is that  the weighted Shannon capacity $ \Theta_G( p)$ is not $\mathfrak{B}^2$ invariant.
\begin{cor}
There exists a graph $G$ such that
$ \Theta_G$ is not $\mathfrak{B}^2$ invariant.
\end{cor}
\begin{proof}
Ac\'in et al have proved~\cite{Acin15} that there exists a graph such that $ \Theta_G$  is not convex.  Hence we infer from Theorem \ref{thm:final} that such a graph provides an example.
\end{proof}

Since the weighted Shannon capacity may not be $\mathfrak{B}^2$ invariant, then we can define the dual  (weigted) Shannon capacity
as
\begin{eqnarray}
\hat{\Theta}_G(p):=(\mathfrak{B}\Theta_{\overline{G}})(p)\;,\;
\hat{\Theta}_{G}:=\hat{\Theta}_G(\textbf{1})\;,
\end{eqnarray}
where the weighted function $\textbf{1}$ assigns 1 to each vertex.
Moreover, we can define the double dual  (weighted) Shannon capacity
as

\begin{eqnarray}
\hat{\hat{\Theta}}_G(p):=(\mathfrak{B}\hat{\Theta}_{\overline{G}})(p)\;,\;
\hat{\hat{\Theta}}_{G}:=\hat{\hat{\Theta}}_G(\textbf{1})\;.
\end{eqnarray}
It is easy to find the relationship between $\alpha_G(p)$,
$\vartheta_G(p)$, $\Theta_G( p)$, $\hat{\Theta}_G( p)$, $\hat{\hat{\Theta}}_G( p)$ and $\alpha^*_G(p)$:

\begin{prop}
If $G=(V(G), E(G))$ is a graph and $p:V(G)\rightarrow [0, +\infty)$ is a weight function on $G$, then
\begin{eqnarray}
\alpha_G(p)\leq  \hat{\hat{\Theta}}_G( p) \leq \Theta_G( p)\leq \vartheta_G(p)\leq \hat{\Theta}_G(p) \leq \alpha^*_G(p)\;.
\end{eqnarray}
\end{prop}

\begin{proof}
It has been proved in \cite{Knuth1994, Acin15} that
\begin{eqnarray}
\alpha_G(p)\leq  \Theta_G(p)\leq \vartheta_G(p)\leq  \alpha^*_G(p)\;.
\end{eqnarray}
Hence, we only need to prove that

\begin{eqnarray}\label{eq:rel1}
 \vartheta_G(p)\leq \hat{\Theta}_G( p) \leq \alpha^*_G(p)\;;
\end{eqnarray}
and
\begin{eqnarray}\label{eq:rel2}
\alpha_G(p)\leq  \hat{\hat{\Theta}}_G( p) \leq \Theta_G(p)\;.
\end{eqnarray}

It is easy to see that Eq. \eqref{eq:rel1} comes from the fact that
$\alpha_{\overline{G}}( w)\leq  \Theta_{\overline{G}}(w)\leq \vartheta_{\overline{G}}(w)$, Proposition \ref{prop:ddual} and
Theorem \ref{thm:g_dual}.
Besides, the first inequality in \eqref{eq:rel2} comes directly from \eqref{eq:rel1}
and Proposition \ref{prop:ddual}, and the second inequality comes  directly from Proposition \ref{prop:ddual}.
\end{proof}

Thus, $\mathfrak{B}$ also provides a way to construct new graph invariants from the old ones.

\begin{prop}If $f,g \in \Delta$, then $tf+(1-t)g\in\Delta $.
That is, $\Delta$ is a convex set.
\end{prop}
\begin{proof}Obviously $tf+(1-t)g \in \mathcal{O}$, and it is monotonically increasing. However,
\begin{eqnarray*} (tf+(1-t)g)(\mu p+(1-\mu)q)
               \leq \mu(tf+(1-t)g)(p)+(1-\mu)(tf+(1-t)g)(q)\;,
\end{eqnarray*}
so $tf+(1-t)g$ is a convex function.
\end{proof}

\begin{prop} If $f=\mathfrak{B}f$, then $f(p)=\sqrt{\inner{p}{p}}$ for every $p\in \real^n_{+}$ .
\end{prop}
\begin{proof}$f(p)=(\mathfrak{B}f)(p)=\sup\limits_{q\neq 0}\frac{\inner{p}{q}}{f(q)}\geq \frac{\inner{p}{p}}{f(p)}$ , so  $f(p)\geq\sqrt{\inner{p}{p}}$ . So\\ $(\mathfrak{B}f)(p)=\sup\limits_{q\neq0}\frac{\inner{p}{q}}{f(q)}\leq \sup\limits_{q\neq0}\frac{\inner{p}{q}}{\sqrt{\inner{q}{q}}}\leq \sqrt{\inner{p}{p}}$ . Therefore $f(p)=\sqrt{\inner{p}{p}}$ .
\end{proof}

\section{Informational interpretation of  graph invariants}
\label{sect:Interpretation}
Max-relative entropy between two quantum states has been introduced in \cite{Dat09a} and plays an important role in the one-shot
manipulation of entanglement and coherence \cite{Dat09b,BD11,Bu17}. Here, we define the max-relative entropy between two
probabilistic models on a hypergraph.

\begin{Def}
For a hypergraph $H=(V(H), E(H))$ and two probabilistic models $p, q\in \mathcal{G}(H)$, the
max-relative entropy of $p$ to $q$ is
\begin{eqnarray}
D_{\max}(p||q)
:=\min\set{\lambda\geq0: p(v)\leq 2^{\lambda} q(v), \forall v\in V(H)}\;.
\end{eqnarray}

\end{Def}

If we choose the probabilistic model $q$ in a subset
$X\subset \mathcal{G}(H)$, then we can quantify the distance between the given
probabilistic model $p$ and the subset $X$ by max-relative entropy as
\begin{eqnarray}
C_{\max}(X, p)
:=\min_{q\in X}D_{\max}(p||q)\;.
\end{eqnarray}
Here the subset $X\subset \mathcal{G}(H)$ is chosen to be
$C(H), Q_1(H), Q(H)$ or $CE_1(H)$. 
Definitions of these four sets can be found in 
\ref{appen:def_con}.

\begin{thm}\label{thm:space}
Given a contextual scenario $H$ and a probabilistic model $p\in \mathcal{G}(H)$, we have the following relationships:
\begin{eqnarray}
\nonumber C_{\max}(C(H), p)&=&\log(\alpha^*_{NO(H)}( p))\geq C_{\max}(Q(H), p)\geq C_{\max}(Q_1(H), p)\\
\nonumber&=&\log(\vartheta_{NO(H)}( p))\geq C_{\max}(CE_1(H), p))=\log(\alpha_{NO(H)}( p))\;.~~~
\end{eqnarray}

\end{thm}
\begin{proof}
The inequalities come from the relationships among the various sets 
$C(H)$, $Q_1(H)$, $Q(H)$, $CE_1(H)$ \cite{Acin15}, namely
\begin{eqnarray*}
C(H)\subset Q(H)\subset  Q_1(H)  \subset CE_1(H)\;.
\end{eqnarray*}
Thus we only need to prove the equalities. These proofs are similar, so we only prove
\begin{eqnarray*}
C_{\max}(C(H), p)=\log(\alpha^*_{NO(H)}(p))\;.
\end{eqnarray*}

Suppose that there are $K$ maximal independent sets
$\set{I_i}$ of $H $, where $K\leq 2^{V(H)}$. Let us define a $K\times |V(H)|$
matrix as:
$M_{iv}=1$ if the vertex $v\in I_i$,
and $M_{iv}=0$ otherwise.

Due to the definition of $C_{\max}(C(H), p)$, $2^{C_{\max}(C(H), p)}$ can be rewritten as the following linear  program (LP):
\begin{eqnarray}\label{LP:1}
\min \sum_i q_i,\quad\text{such that }
~Mq\geq p\;, \quad\text{and } 
q\geq 0\;.
\end{eqnarray}
 The dual LP can be written as:
 \begin{eqnarray}\label{LP:2}
\max \sum_v p(v)w(v)\;,\quad\text{such that }
~M^Tw\leq 1\;,\quad\text{and }
w\geq 0\;.
\end{eqnarray}
 It is easy to verify that
 the strongly feasible condition holds which implies that
 \eqref{LP:1} is equal to \eqref{LP:2} \cite{DT2003}. That is, $2^{C_{\max}(C(H), p)}$
 can be expressed as \eqref{LP:2}.

 Due to the one-to-one correspondence between the independent sets in $H$ and
 cliques in $NO(H)$, the condition $M^Tw\leq 1$ means that
 $\sum\limits_{v\in C}w(v)\leq 1$ for any clique $C$ in $NO(H)$.
According to the definition of
$\alpha^*_{NO(H)}(p)$,
\begin{eqnarray*}
\alpha^*_{NO(H)}(p)
=\max_{q}\sum_vp(v)q(v)\;,
\end{eqnarray*}
where the weight function $q: V(H)\to \real_+$ satisfies that
$\sum\limits_{v\in C}q(v)\leq 1$ for any clique $C$ in $NO(H)$.
Thus, we have
$2^{C_{\max}(C(H), p)}=\alpha^*_{NO(H)}(p)$ .
\end{proof}

Any nontrivial inequality for a contextual scenario
$(V(H), E(H)) $ is represented by a real weight function
$w: V(H)\rightarrow \real$  with at least one non-negative component.
Given a probabilistic model $p$ on the contextual scenario
$H$, the inequality
\begin{eqnarray}
\langle p, w \rangle
\leq L(H, w)
\end{eqnarray}
is called a generalized Bell inequality if it is satisfied by all classical
models $p_{cl}\in C(H)$.
(See \cite{Brunner2014} for the details of a Bell inequality.)
Let us take $L(H, w)=\max\limits_{p_{cl}\in C(H)} \langle p_{cl}, w \rangle$
where the maximization is taken over all classical models.
Then the violation of the generalized  Bell inequality $w$ by the probabilistic model
$p$
is quantified by
\begin{eqnarray}
 \frac{\inner{p}{w}}{L(H, w)}\;.
\end{eqnarray}
\begin{thm}[\bf Physical interpretation of duality]\label{thm:phy}
Given a contextual scenario $H$ and a probabilistic model $p$,
the $\alpha^*_{NO(H)}(p)$ provides a lower bound on  the maximum violation of all positive
Bell-type inequalities of the probabilistic model, namely,
\begin{eqnarray}
\alpha^*_{NO(H)}(p)
\leq \max_{w\geq 0} \frac{\inner{p}{w}}{L(H, w)}\;.
\end{eqnarray}

\end{thm}

\begin{proof}
Using Theorem \ref{thm:g_dual}, we only need to show that
$L(H, w)\leq \alpha_H( w)$.
Since each classical model
can be written as a convex combination of
deterministic models,
then
maximization in $L(H, w)$ can be realized by deterministic models,
i.e., $L(H, w)=\max\limits_{p~ \text{deterministic}}\inner{p}{w}$.
It has been proved in \cite{Acin15} that
each  deterministic model $p$ is
characterized
 by the set
\begin{eqnarray}
 V_p
 =\set{v\in V|\ p(v)=1}\;,
 \end{eqnarray}
where $V_p$ intersects each edge in
exactly one vertex. Thus,
$V_p$ is a maximal independent set and
$\inner{p}{w}=\sum\limits_{v\in V_p}w(v)$.
Therefore, 
$L(H, w)\leq \alpha_H(w)$.
\end{proof}

\section{Conclusion and discussion }
In this paper we define a transform $\mathfrak{B}$ on functions, and using $\mathfrak{B}$ we propose a new duality between some graph invariants: these include  the weighted independence number, weighted
Lov\'asz number, and weighted fractional-packing number. We find that they are all $\mathfrak{B}^2$ invariant.
We also find necessary and sufficient conditions for a function to be $\mathfrak{B}^2$ invariant, which only requires it to be convex and monotonically increasing. The transformation 
$\mathfrak{B}$  gives  a new approach to construct  graph invariants.  Moreover, as graph invariants play an important role in the investigation of
contextuality, we provide a new informational interpretation of these weighted graph invariants by max-relative entropy in the contextual scenario. These results shed new insight on  graph theory and quantum information theory.

In addition to the graph invariants we consider in this work, there are also other interesting and useful graph invariants, such as
the chromatic number, which is the minimal number of colors to color the vertices of a graph with the vertices of any edge being different colors.  One could analyze these as well.

It is noteworthy that $\mathfrak{B}$ is very similar to the Legendre-Fenchel transformation \cite{Fenchel49}. Changing the addition in the definition of the Legendre-Fenchel transformation into multiplication then we can get $\mathfrak{B}$. The two transformations have similar properties: as they map to convex functions and are double dual invariant on convex functions. 
However, $\mathfrak{B}$  cannot be obtained simply by taking logarithm of the Legendre-Fenchel transformation; the techniques used in the proof of the Legendre-Fenchel transformation seem not directly applicable  in our proof. This may mean that there is  some general transformation, with these two transformations as special cases.

\section*{Acknowledgment}
 The authors thank Zhengwei Liu for his support, help, and comments. W. Gu thanks Liming Ge and Boqing Xue for discussion on related material.  
This research was supported  by the Templeton Religion Trust under grant TRT 0159.  A.~Jaffe was also supported in part by the ARO  Grant W911NF1910302. K.~Bu is grateful for the support of an Academic Award for Outstanding Doctoral
Candidates from Zhejiang University.

\begin{appendix}
\section{Definitions Relating to Contextuality}\label{appen:def_con}
We repeat the definitions of a number of concepts in Ac\'in et al \cite{Acin15}, from the point of view of this paper.
\begin{Def}[\bf Hypergraph \rm or \bf Contextual senario]
A hypergraph is a pair $H=(V(H),E(H))$, where $V(H)$ is a set, and $E(H)$ is a set of non-empty subsets of $V(H)$. Elements in $V(H)$ are  the vertices of $H$, and elements in $E(H)$ are the hyperedges of $H$.
\end{Def}

\begin{Def}[\bf Non-orthogonality graph]
Given a contextual scenario $H$, the
non-orthogonality graph NO(H) is the graph with the set of vertices $V(NO(H))=V(H)$ and the
adjacency relations $u\sim v$ in $NO(H)$ iff there is no hyperedge $e$ such that $u\in e, v\in e$.
\end{Def}

\begin{Def}[\bf Probabilistic model]
Given a contextual scenario $H$,  a probabilistic model on $H$ is a function $p:V(G)\rightarrow [0, +\infty)$ such that
\begin{eqnarray}
\sum_{v\in e}p(v)=1\;,\;\; \forall e\in E(H)\;.
\end{eqnarray}
The set of all probabilistic models on $H$ is denoted as $\mathcal{G}(H)$.
\end{Def}

\begin{Def}[\bf Classical model]
Given a contextual scenario $H$, a probabilistic model $p\in \mathcal{G}(H)$ is called deterministic if
for any $v\in V(H)$, $p(v)=0$ or $1$.  Moreover, a probabilistic model $p\in \mathcal{G}(H)$ is called classical if it can be
written as a convex combination of some deterministic probabilistic models.  Denote the set of all classical probabilistic models
by $C(H)$.
\end{Def}

\begin{Def}[\bf Quantum model]
Let $H$ be a contextual scenario, then a probabilistic model $p\in \mathcal{G}(H)$  is called a quantum model if there exists
a Hilbert space $\mathcal{H}$, a quantum state $\rho\in D(\mathcal{H}) $ (i.e., $\rho\geq0$ and $\mathrm{Tr}\rho=1$ ) and a projection measurement $\set{P_v}_{v\in e}$ for
each hyperedge $e$ such that
\begin{eqnarray}
\sum_{v\in e}P_v=\mathbb{I}_{\mathcal{H}}\;,\;
p(v)=\mathrm{Tr}(P_v\rho), \forall v\in V(H)\;.
\end{eqnarray}
The set of all quantum models on $H$ is denoted as $Q(H)$.
\end{Def}

\begin{Def}[\bf $Q_1$ model ]
Given a contextual scenario $H$, let $Q_1$ be the set of probabilistic models $p\in \mathcal{G}(H)$ satisfying:
there exists a Hilbert space $\mathcal{H}$, a normalized vector $\ket{\Psi}\in \mathcal{H}$ and
a set of normalized vectors $\set{\ket{\psi_v}}_{v\in V(H)}$ such that
(i) $\ket{\psi_u}\bot\ket{\psi_v}$ for any $u, v\in V(H)$ which are not adjacent in $H$;
(ii) $p(v)=|\langle\psi_v|\Psi\rangle|^2$ for any $v\in V(H)$.
\end{Def}

\begin{Def}[\bf $CE_1$ model]
Given a contextual scenario $H$, we say a probabilistic model $p\in \mathcal{G}(H)$ satisfies the Consistent Exclusivity
if $\sum_{v\in I}p(v)\leq 1$ for any independent set $I\subset V(NO(H))$. The set of such probabilistic models is denoted as
$CE_1(H)$.
\end{Def}

\end{appendix}

\end{document}